\documentclass[a4paper,11pt]{article}
\usepackage{a4wide}
\usepackage[UKenglish]{isodate}
\usepackage[utf8]{inputenc}
\usepackage[T1]{fontenc}

\usepackage{latexsym,amsthm,amsmath,amssymb,amsfonts,mathtools,float,cite,derivative,doi}
\usepackage{hyperref}

\newtheorem{thm}{Theorem}[section] 
\newtheorem{mythm}[thm]{Theorem}
\newtheorem{mycor}[thm]{Corollary}
\newtheorem{mylemma}[thm]{Lemma}
\newtheorem{myprop}[thm]{Proposition}
\newtheorem{myques}[thm]{Question}
\newtheorem{myconj}[thm]{Conjecture}
\newtheorem{myobs}[thm]{Observation}

\newtheoremstyle{TheoremNum}
    {\topsep}{\topsep}              
    {\itshape}                      
    {}                              
    {\bfseries}                     
    {.}                             
    { }                             
    {\thmname{#1}\thmnote{ \bfseries #3}}
\theoremstyle{TheoremNum}
\newtheorem{repthm}{Theorem} 

\allowdisplaybreaks

\DeclareMathOperator{\lexp}{\bullet}

\cleanlookdateon

\title{\bf Multi-Colouring of Kneser Graphs:\\[1mm]
Notes on Stahl's Conjecture\,\footnote{The research in this publication was first published in the second author's PhD Thesis~\cite{xuPhD}.}}

\author{Jan van den Heuvel\qquad and\qquad Xinyi Xu\\[3mm]
    Department of Mathematics\\[0.5mm]
    London School of Economics \& Political Science\\[0.5mm]
    Houghton Street, London WC2A 2AE, UK\\[3mm]
    j.van-den-heuvel@lse.ac.uk\qquad xinyi.xu97@gmail.com}

\begin{document}
\maketitle

\begin{abstract}
\noindent
A (finite, undirected) graph is \emph{$(n,k)$-colourable} if we can assign each vertex a $k$-subset of $\{1,2,\ldots,n\}$ so that adjacent vertices receive disjoint subsets.
We consider the following problem: if a graph is $(n,k)$-colourable, then for what pairs $(n', k')$ is it also $(n',k')$-colourable?
This question can be translated into a question regarding multi-colourings of Kneser graphs, for which Stahl formulated a conjecture in 1976.
We present new results, strengthen existing results, and in particular present much simpler proofs of several known cases of the conjecture.
\end{abstract}
\section{Introduction and Main Results}\label{multi_sec:intro}

All graphs in this note are finite, undirected and without multiple edges or loops.
Our terminology and notation is standard, and can be find in any textbook on graph theory such as Diestel~\cite{diestelgraph}.

A \emph{proper} colouring of a graph assigns a colour to each vertex such that adjacent vertices receive different colours.
A graph $G$ is \emph{$n$-colourable} if~$n$ colours are enough for a proper colouring of~$G$, and the \emph{chromatic number $\chi(G)$} is the smallest~$n$ for which $G$ is $n$-colourable.

\emph{Multi-colouring} generalises vertex colouring and itself has been the subject of extensive research; see e.g.\ \cite[Chapter~3]{scheinerman2011fractional}.
In a \emph{$k$-multi-colouring of a graph}, each vertex receives a set of $k$ colours, and such a colouring is \emph{proper} if adjacent vertices receive disjoint colour sets.
A graph $G$ is \emph{$(n, k)$-colourable} if there is a proper $k$-multi-colouring using $k$-subsets from~$[n]$ (${}=\{1,2,\ldots,n\}$).
For a positive integer $k$, the \emph{$k$-th multi-chromatic number~$\chi_k(G)$} is the smallest~$n$ such that $G$ is $(n,k)$-colourable. 

Note that if $k=1$, then $k$-multi-colouring is just normal vertex colouring, and $\chi_1(G)$ is just the normal chromatic number $\chi(G)$.

It this note we consider the following question.

\begin{myques}\label{main_ques}\mbox{}\\*
    If a graph $G$ is $(n, k)$-colourable, then for what pairs $(n', k')$ are we guaranteed that~$G$ is also $(n', k')$-colourable?
\end{myques}

Note that the corresponding question for standard $n$-colouring is trivial: if~$G$ is $n$-colourable, then it is $n'$-colourable for all $n'\ge n$.
More precisely: if $\chi(G)=n$, then~$G$ is $n'$-colourable if and only if $n'\ge n$.
Maybe somewhat surprisingly, the question for multi-colouring appears to be much more challenging, and in fact is mostly open.

\medskip
Kneser graphs play a central role in the studies of multi-colouring.
For $n\ge k\ge1$, the \emph{Kneser graph $K(n, k)$} has as vertex set the collection of all $k$-subsets of $[n]$ (denoted by $\dbinom{[n]}{k}$), and there is an edge between two vertices if and only if the two $k$-sets are disjoint.
We will usually assume $n\ge2k$, as otherwise the Kneser graph is edgeless.

It is well known and easy to prove (see e.g.\ \cite[Section~3.2]{scheinerman2011fractional}) that a graph~$G$ is $(n,k)$-colourable if and only if there is a homomorphism from~$G$ to $K(n,k)$.
(A \emph{homomorphism} from a graph~$G$ to a graph $H$ is a mapping $\varphi:V(G)\to V(H)$ that preserves edges: if $uv$ is an edge in~$G$, then $\varphi(u)\varphi(v)$ is an edge in $H$.)
This means that the following question is equivalent to Question~\ref{main_ques}.

\begin{myques}\label{main_ques2}\mbox{}\\*
    Given $n,k$, for what $n',k'$ do we have $n'\ge\chi_{k'}(K(n,k))$?
\end{myques}

This question was first studied by Stahl, who formulated the following conjecture in 1976.

\begin{myconj}[Stahl~\cite{stahl1976n}]\label{st-con}\mbox{}\\*
    Let $n,k$ be integers, $n\ge2k\ge2$.
    Then for $k'=qk-r$, where $q,r$ are integers with $q\ge1$ and $0\le r\le k-1$, we have $\chi_{k'}(K(n,k)) = qn-2r$.
\end{myconj}

The conjecture is known to hold for some special values of $n,k,k'$, but few general results are known. For instance, the conjecture is trivially true for $k=1$ (since the Kneser graph $K(n,1)$ is just the complete graph on $n$ vertices).
It is also true for $k'=1$ by Lov\'asz's proof~\cite{lovasz1978kneser} of the Kneser Conjecture: $\chi_{1}(K(n,k))= \chi(K(n,k))=n-2(k-1)$.
(Though note that at the time~\cite{stahl1976n} appeared the Kneser Conjecture was still open.)

In Stahl's original paper it was proved that the conjectured value is an upper bound, i.e.\ ($n,k,q,r$ as in Conjecture~\ref{st-con}):
\begin{equation}\label{eq:up_bd}
    \chi_{qk-r}(K(n,k))\le qn-2r.
\end{equation}
In the next section we explain how this bound can be derived.

Our first results are based on a simple observation, which seems to have been missed in the research on Stahl's Conjecture.

\begin{myobs}\label{obs1}\mbox{}\\*
    Let $n,k,k'$ be integers, $n\ge2k\ge2$, $k'\ge1$. Then we have $\chi_{k'}(K(n,k))\ge \dfrac{k'n}{k}$.
\end{myobs}

\begin{proof}
For any $(n',k')$-colouring of a graph~$G$, each \emph{colour class} (the set of vertices whose colour set contains a particular colour) is an independent set, hence contains at most $\alpha(G)$ vertices.
(Here $\alpha(G)$ is the \emph{independence number of $G$}, the size of a largest independent set in $G$.)
Since each vertex appears in $k'$ colour classes, we have $k'\cdot|V(G)|\le n'\cdot\alpha(G)$, hence $\chi_{k'}(G)\ge\dfrac{k'|V(G)|}{\alpha(G)}$.

For Kneser graphs, by definition we have $|V(K(n,k))|=\dbinom{n}{k}$, while the celebrated Erd\H{o}s-Ko-Rado Theorem~\cite{erdos_intersection_1961} means that $\alpha(K(n,k))=\dbinom{n-1}{k-1}$ for all $n\ge2k\ge2$.
Substituting those values in the lower bound for $\chi_{k'}(G)$ above immediately gives $\chi_{k'}(K(n,k))\ge \dfrac{k'{\binom{n}{k}}}{\binom{n-1}{k-1}}= \dfrac{k'n}{k}$.
\end{proof}
 
Since $\chi_{k'}(K(n,k))$ is an integer, we immediately obtain the following result.

\begin{mythm}\label{thm-main}\mbox{}\\*
    Let $n,k$ be integers, $n\ge2k\ge2$. 
    Then for $k'=qk-r$, where $q,r$ are integers with $q\ge1$ and $0\le r\le k-1$, we have $\chi_{k'}(K(n,k))\ge \Bigl\lceil\dfrac{k'n}{k}\Bigr\rceil= qn-\Bigl\lfloor\dfrac{rn}{k}\Bigr\rfloor= qn-2r-\Bigl\lfloor\dfrac{r(n-2k)}{k}\Bigr\rfloor$.
\end{mythm}

Surprisingly, by splitting the Kneser graph $K(n,k)$ into smaller subgraphs and using Theorem~\ref{thm-main} for each of the subgraphs, sometimes it is possible to get a better bound than using the theorem directly. 

\begin{thm}\label{thm-split}\mbox{}\\*
    Let $n,k,r,q$ be integers, $n\ge2k\ge2$, $q\ge1$ and $0\le r\le k-1$.
    Choose $n_1,\ldots,n_t$ such that $n=\sum\limits_{i=1}^tn_i$ and $2k\le n_i<4k$ for all~$i$.
    Then we have $\chi_{qk-r}(K(n,k))\ge qn-\sum\limits_{i=1}^t\Bigl\lfloor\dfrac{n_ir}{k}\Bigr\rfloor$.
\end{thm}

We will prove this theorem in Subsection~\ref{ssec3.1}.
And although the result is a fairly direct corollary of Theorem~\ref{thm-main}, it actually can give better bounds in many cases.
For example, if $n=t(2k+1)$, then Theorem~\ref{thm-split} with $n_1=\cdots=n_t=2k+1$ gives $\chi_{qk-r}(K(n,k))\ge qn-2rt$, whilst Theorem~\ref{thm-main} only gives $\chi_{qk-r}(K(n,k))\ge qn-2rt-\Bigl\lfloor\dfrac{rt}k\Bigr\rfloor$.

\medskip
Observation~\ref{obs1} also almost immediately gives the following small improvement of the main result in Oszt\'enyi~\cite{osztenyi2020proof}; again with a much simpler and shorter proof.
As before, details can be found in Subsection~\ref{ssec3.1}.

\begin{mythm}\label{thm:main1}\mbox{}\\*
    Let $n,k,r$ be integers, $n\ge2k+1\ge3$ and $0\le r\le\dfrac{k}{n-2k}$. Then for all $q\ge1$ we have $\chi_{qk-r}(K(n,k))=qn-2r$.
\end{mythm}

\medskip
Our next result shows that for a fixed $k$ only a finite number of values of $\chi_{k'}(K(n,k))$ need to be determined in order to conclude whether or not Stahl's Conjecture is true for that value of $k$ and for \emph{all $n$ and $k'$}.

\begin{mythm}\label{thm:main2}\mbox{}\\*
    Let $k\ge2$ be fixed.
    Then there exist $n_0(k)$ and $q_0(n,k)$ such that the following holds.
    If for all $2k\le n\le n_0(k)$ we know that $\chi_{qk-(k-1)}(K(n,k))=qn-2(k-1)$ for at least one $q\ge q_0(n,k)$, then we have $\chi_{qk-r}(K(n,k))=qn-2r$ for all $n\ge2k$, $q\ge1$ and $0\le r\le k-1$.
\end{mythm}

Possible functions $n_0(k)$ and $q_0(n,k)$ in Theorem~\ref{thm:main2} are given explicitly in Section~\ref{multi_sec:proofs}.
We use these expressions to show that we can take $n_0(k)\le k^3-k^2+2k-2$ and $q_0(n,k)<4^k(n-2k)$ for all $n\ge2k+1\ge5$.

We could replace $q_0(n,k)$ by $q'_0(k)=\max\{\,q_0(n,k)\mid2k\le n\le n_0(k)\,\}$ in the theorem, to remove the dependency of $q_0$ on $n$.
We chose to keep $q_0(n,k)$, since for larger values of $n$ we get better bounds for $q_0(n,k)$.
For instance, if $n\ge k^2+k-1$, then we can show $q_0(n,k)<\mathrm{e}(n-2k)$.

Theorem~\ref{thm:main2} generalises some known results. 
Chv\'atal et al.~\cite{chvacutetal1978two} showed that for fixed $k$ we only need to find $\chi_{k+1}(K(n,k))$ for finitely many~$n$ to decide if Stahl's Conjecture holds for $k'=k+1$ for all~$n$.
And Stahl~\cite{stahl1976n} proved that for fixed $n,k$ and sufficiently large $k'$, the conjecture holds for $k'$ if and only if it holds for $k'-k$.
The proof of that latter result is non-constructive and does not give an explicit bound on the value of $k'$, and hence it can only give a version of Theorem~\ref{thm:main2} without a bound on the function $q_0(n,k)$.

\medskip
Stahl proved the conjecture for $k=2,3$ (see next section).
For $k=4$, our methods show that we only need to find $\chi_{4q-3}(K(n,4))$ for $8\le n\le10$ and $q=13$, and for $11\le n\le 39$ and $q=12$.
(In fact, for larger~$n$ even smaller~$q$ are enough.)
The cases $n=8,9$ follow from Theorem~\ref{th1}\,(a) below.
The case $n=10$ is solved in~\cite{kincses2013special}.
So the first open case is to determine whether or not $\chi_{45}(K(11,4))=126$.
Note that Stahl's bound~\eqref{eq:up_bd} gives $\chi_{45}(K(11,4))\le 126$, while Theorem~\ref{thm-main} shows $\chi_{45}(K(11,4))\ge 124$.

\medskip
The remainder of this paper is organised as follows.
In the next section we discuss some of Stahl's own work in more detail. 
The proofs of our results can be found in Section~\ref{multi_sec:proofs}.
Final observations are discussed in Section~\ref{multi_sec:conclude}.

\section{Stahl's Work} \label{multi_sec:ideas}

In this section we describe some of the original ideas behind Stahl's Conjecture, as developed in \cite{stahl1976n,stahl1998multichromatic}, since many of these ideas are important in the development of our arguments.

The following two results are essential in showing that the conjectured values of $\chi_{k'}(K(n,k))$ are indeed upper bounds.

\begin{mylemma}[Geller and Stahl~\cite{geller1975chromatic}]\label{prop2.1}\mbox{}\\*
    If a graph $G$ is both $(n_1,k_1)$-colourable and $(n_2,k_2)$-colourable, then~$G$ is $(n_1+n_2,k_1+k_2)$-colourable.
\end{mylemma}

\begin{mylemma}[Stahl~\cite{stahl1976n}]\label{sta:prop}\mbox{}\\*
    For integers $n\ge3$ and $k\ge2$, there exists a homomorphism $\varphi$ from $K(n,k)$ to $K(n-2, k-1)$.
\end{mylemma}

The proof of Lemma~\ref{prop2.1} combines an $(n_1,k_1)$-colouring using colours from $[n_1]=\{1,\ldots,n_1\}$ and an $(n_2,k_2)$-colouring using colours from $\{n_1+1,\ldots,n_1+n_2\}$.

\medskip
The proof of Lemma~\ref{sta:prop} in~\cite{stahl1976n} is more cumbersome than required, so we give a new proof here. First note that if $n<2k$, then any mapping of vertices is a homomorphism, because both $K(n,k)$ and $K(n-2,k-1)$ are edgeless.

So we can assume $n\ge2k$. For each $k$-subset $S\subseteq[n]$, let $\max S$ be the maximum element of~$S$.
Now for each $k$-subset~$S$ with $|S\cap\{n-1,n\}|\le1$ we set $\varphi(S)=S\setminus\{\max S\}$.
If $\{n-1,n\}\subseteq S$, then let~$x$ be the largest integer in $[n]\setminus S$ and set $\varphi(S)=(S\setminus\{n-1,n\})\cup\{x\}$.

To show that $\varphi$ is a homomorphism from $K(n,k)$ to $K(n-2,k-1)$, we must show that if any two $k$-sets $S_1,S_2$ are adjacent in $K(n,k)$, hence are disjoint, then $\varphi(S_1)$ and $\varphi(S_2)$ are disjoint as well.
This is obvious if both $|S_1\cap\{n-1,n\}|\le1$ and $|S_2\cap\{n-1,n\}|\le1$, since then $\varphi(S_1)\subseteq S_1$ and  $\varphi(S_2)\subseteq S_2$, hence $\varphi(S_1)\cap\varphi(S_2)\subseteq S_1\cap S_2=\varnothing$. 
If for one of $S_1,S_2$, say~$S_1$, we have $\{n-1,n\}\subseteq S_1$, then for~$S_2$ we must have $|S_2\cap\{n-1,n\}|=0$.
Since $\max S_2\notin S_1$, we have that the largest element~$x$ in $[n]\setminus S_1$ satisfies $x\ge\max S_2$.
This means that $(S_1\cup\{x\})\cap(S_2\setminus\{\max S_2\})=\varnothing$, which guarantees that $\varphi(S_1)$ and $\varphi(S_2)$ are also disjoint in this case.\hfill\qedsymbol

\medskip
The existence of a homomorphism from $K(n,k)$ to $K(n-2,k-1)$ means that for any graph~$G$, if~$G$ is $(n,k)$-colourable, then $G$ is also $(n-2,k-1)$-colourable.
And hence for any graph $G$ with at least one edge, we have $\chi_{k'-1}(G)\le\chi_{k'}(G)-2$. (For an edge-less graph $G$ we have $\chi_{k'-1}(G)=k'-1=\chi_{k'}(G)-1$ for all~$k'$.)

Stahl's Conjecture states that $\chi_{qk-r}(K(n,k))=qn-2r$, for all $q\ge1$ and $0\le r\le k-1$.
The two lemmas above can explain where this expression comes from. Firstly, by definition $K(n,k)$ is $(n,k)$-colourable.
By using Lemma~\ref{prop2.1} $q\ge1$ times, we find that $K(n,k)$ is $(qn,qk)$-colourable.
(And then in fact Observation~\ref{obs1} shows that $\chi_{qk}(K(n,k))=qn$.)
Now applying Lemma~\ref{sta:prop} $r\ge0$ times gives that $K(n,k)$ is $(qn-2r,qk-r)$-colourable.
This proves the upper bound $\chi_{qk-r}(K(n,k))\le qn-2r$ from~\eqref{eq:up_bd}.

Combining that $\chi_{qn}(K(n,k))=qn$ and $\chi_{k'-1}(K(n,k))\le\chi_{k'}(K(n,k))-2$ for all~$k'$, we can conclude that for any fixed $q$, if $\chi_{qk-(k-1)}(K(n,k))=qn-2(k-1)$, then $\chi_{qk-r}(K(n,k))=qn-2r$ for the same~$q$ and all $0\le r\le k-1$.
This means that in order to prove Stahl's Conjecture, it suffices to prove it for $r=k-1$.
An immediate corollary of the Lov\'asz-Kneser Theorem, $\chi_1(K(n,k))=n-2k+2$, is that the conjecture is true for $q=1$; see Theorem~\ref{th1}(c).

The observations in this section provide other ways to prove the lower bound $\chi_{qk-r}(K(n,k))\le qn-2r$.
For instance, we can use Lemma~\ref{prop2.1} to construct a $(qn-2r,qk-r)$-colouring of $K(n,k)$ by combining $q-1$ copies of an $(n,k)$-colouring and one copy of an $(n-2r,k-r)$ colouring.
And in fact, in general there are many ways to obtain similar multi-colourings.
For instance, if $k,q,r\ge2$, we can take $q-2$ copies of an $(n,k)$-colouring, one copy of an $(n-2,k-1)$-colouring and one copy of an $(n-2(r-1),k-(r-1))$-colouring to get the same bound.

This multitude of possible $(n-2(r-1),k-(r-1))$-colourings may be one of the reasons why Stahl's Conjecture is so difficult to prove.
The proof of Theorem~\ref{thm:eq_indep_bound} below shows that taking~$q$ copies of an $(n,k)$-colouring of $K(n,q)$ is in essence the only way to obtain an $(qn,qk)$-colouring of $K(n,q)$, which might explain why for that case we can prove the conjecture.

\medskip
Stahl also proved the conjecture for some special values of $n,k,k'$.

\pagebreak[4]
\begin{mythm}\label{th1}\mbox{}\\*
    {\rm (a)} For all $k$ and $k'$, Conjecture \ref{st-con} is true for the bipartite Kneser graphs $K(2k,k)$ and for the so-called odd graphs $K(2k+1,k)$ (Stahl~\cite{stahl1976n}).

    {\rm (b)} For all $n$ and $k$, Conjecture \ref{st-con} is true for any $k'$ that is an integer multiple of~$k$; in other words: $\chi_{qk}(K(n,k))=qn$ (Stahl~\cite{stahl1976n}).

    {\rm (c)} For all $n$ and $k$, Conjecture \ref{st-con} is true for all $k'\le k$; in other words: $\chi_{k-r}(K(n,k))=n-2r$ (Stahl~\cite{stahl1976n}).

    {\rm (d)} For all $n$ and $k'$, Conjecture \ref{st-con} is true for $k=2$ and $k=3$ (Stahl~\cite{stahl1998multichromatic}).
\end{mythm}

As remarked earlier, combining Theorem~\ref{thm-main} with~\eqref{eq:up_bd} gives a proof of Theorem~\ref{th1}\,(b) that is much simpler than the proof in~\cite{stahl1976n}.

\medskip
We can obtain further results by using more detailed knowledge about independent sets in Kneser graphs.
Erd\H{o}s, Ko and Rado~\cite{erdos_intersection_1961} proved that if $n\ge2k+1$, then the only independent sets of order $\dbinom{n-1}{k-1}$ in the Kneser graph $K(n,k)$ are the so-called \emph{trivial} independent sets: those vertex sets whose vertices correspond to family of $k$-sets in~$[n]$ that contain some fixed common element~$i\in[n]$.
Using that information about the structure of independent sets of order $\alpha(K(n,k))$, the following can be proved.

\begin{thm}\label{thm:eq_indep_bound}\mbox{}\\*
    Let $n,k$ be integers, $n\ge2k+1\ge3$. Then for all $k'\ge1$ we have equality in Observation~\ref{obs1} if and only if $k'$ is an integer multiple of $k$.
\end{thm}

The statement that $\chi_{pk}(K(n,k))=pn$ if and only if~$p$ is an integer is sometimes attributed to Stahl~\cite{stahl1976n} (see, e.g., \cite[Section~7.9]{godsil2001algebraic} and \cite{osztenyi2020proof}), but it is not explicitly stated in Stahl's paper.
Stahl's result \cite[Theorem~9]{stahl1976n} states that $K(n,k)$ has a \emph{particular} $(pn,pk)$-colouring (called ``efficient'' in the paper) if and only if~$p$ is an integer.  Nevertheless, its proof in essence proves the general result, as is made explicit in \cite[Lemma~7.9.3]{godsil2001algebraic} and its proof.

\medskip
In a follow-up paper~\cite{stahl1998multichromatic}, Stahl proved the following general lower bound for $\chi_{k'}(K(n,k))$:
\begin{equation}\label{eq2}
    \chi_{qk-r}(K(n,k))\ge qn-2r-(k^2-3k+4).
\end{equation}

The estimate in Theorem~\ref{thm-main} is better than~\eqref{eq2} if $n\le k^2+2$.

\medskip
As already observed by Geller \& Stahl~\cite{geller1975chromatic}, for any graph $G$ we have $\chi_{k'}(G) = \chi(G\lexp K_{k'})$, where ``$\bullet$'' denotes the \emph{lexicographic product} of two graphs: $V(G\lexp H) = V(G)\times V(H)$, and $(u_1,v_1)(u_2,v_2)\in E(G\lexp H)$ if and only if either $u_1u_2\in E(G)$, or $u_1=u_2$ and $v_1v_2\in E(H)$.
This allows us to translate the problem of finding multi-chromatic numbers to finding chromatic numbers. 
Since we also have that $|V(G\lexp K_{k'})|=k'|V(G)|$ and $\alpha(G\lexp K_{k'})=\alpha(G)$, this gives an alternative proof of $\chi_{k'}(G)\ge\dfrac{k'|V(G)|}{\alpha(G)}$.

\section{Proofs of Our Results} \label{multi_sec:proofs}

This section contains the proofs of the results from Section~\ref{multi_sec:intro}.
Throughout this section we use~$k'$ and $qk-r$ interchangeably (i.e.\ $k'=qk-r$), where $q\ge 1$ and $0\le r\le k-1$.

\subsection{Proof and Discussion of Theorems~\ref{thm-split} and~\ref{thm:main1}}\label{ssec3.1}

We first prove Theorem \ref{thm:main1}, which states that Stahl's Conjecture is true for $0\le r \le \dfrac{k}{n-2k}$. Note that this result generalises the following known results.

(a) \textit{Stahl's Conjecture is true if $k'$ is a multiple of $k$, i.e.\ if $r=0$ (Stahl~\cite{stahl1976n}).}

(b) \textit{Stahl's Conjecture is true if $2k<n<3k$ and $0\le r<\dfrac{k}{n-2k}$ (Oszt\'enyi~\cite{osztenyi2020proof})}.

\begin{repthm}[\ref{thm:main1}]\mbox{}\\*
Let $n,k,r$ be integers, $n\ge2k+1\ge3$ and $0\le r\le\dfrac{k}{n-2k}$. Then for all $q\ge1$ we have $\chi_{qk-r}(K(n,k))= qn-2r$.
\end{repthm}

\begin{proof}
If $0\le r<\dfrac{k}{n-2k}$, then $\Bigl\lfloor\dfrac{r(n-2k)}{k}\Bigr\rfloor=0$, and hence Theorem~\ref{thm-main} immediately shows that $\chi_{qk-r}(K(n,k))\ge qn-2r$. We then have $\chi_{qk-r}(K(n,k))=qn-2r$ by~\eqref{eq:up_bd}.

If $r=\dfrac{k}{n-2k}$ is an integer and $0<r\le k-1$, then $\dfrac{k'n}{k}= \dfrac{(qk-r)n}{k}= qn-\dfrac{n}{n-2k}= qn-2r-1$ is an integer. 
But as $qk-r$ is not a multiple of~$k$, since $1\le r\le k-1$, by Theorem~\ref{thm:eq_indep_bound} we have that $K(n,k)$ is not $(qn-2r-1, qk-r)$-colourable.
We can conclude that $\chi_{qk-r}(K(n,k))= qn-2r$ by~\eqref{eq:up_bd}.
\end{proof}

The result in~\cite{osztenyi2020proof} mentioned above follows from the following lower bound in that paper, obtained after a long and quite elaborate proof.

\begin{thm}[Oszt\'enyi {\cite[Proposition 5]{osztenyi2020proof}}]\label{osz-thm}\mbox{}\\*
Let $n,k,\ell$ be integers, $k,\ell\ge2$ and $\ell k<n<2\ell k$.
Then for all $q\ge 1$ and $r$, $0\le r\le k-1$, we have $\chi_{qk-r}(K(n,k))\ge qn-\ell r-c+1$, where $c$ is a positive integer satisfying $c>\dfrac{\ell r-1}{\lceil\frac{\ell k}{n-\ell k}\rceil-1}$.
\end{thm}

We will show that the bound in Theorem~\ref{thm-main} is already at least as good as this bound, by proving that we never have $qn-\ell r-c+1>qn-2r-\Bigl\lfloor\dfrac{r(n-2k)}{k}\Bigr\rfloor$ for integers $n,k,\ell,c$ satisfying the conditions in Theorem~\ref{osz-thm}.
Since $\ell\ge2$, it is more than enough to show that for $\ell k< n<2\ell k$ there is no positive integer~$c$ such that $\dfrac{\ell r-1}{\lceil\frac{\ell k}{n-\ell k}\rceil -1}<c<\Bigl\lfloor\dfrac{r(n-2k)}{k}\Bigr\rfloor+1$.
For a contradiction, assume such a $c$ exists for some $n,k,r,\ell$.
Then since $c$ is an integer, we have
\begin{equation}\label{eq:compare_osz_2}
    ck\le r(n-\ell k).
\end{equation}
Since $n<2\ell k$ is equivalent to $n-\ell k\le\ell k-1$ for integers, we have
$\Bigl\lceil\dfrac{\ell k}{n-\ell k}\Bigr\rceil-1\ge \Bigl\lceil\dfrac{\ell k}{\ell k-1}\Bigr\rceil- 1 >0$.
Using this in $\dfrac{\ell r-1}{\lceil\frac{\ell k}{n-\ell k}\rceil -1}<c$ gives $c\Bigl\lceil\dfrac{\ell k}{n-\ell k}\Bigr\rceil-c>\ell r-1$.
We can rearrange this to $\Bigl\lceil \dfrac{\ell k}{n-\ell k}\Bigr\rceil> \dfrac{\ell r+c-1}{c}$, and hence to $\Bigl\lceil \dfrac{\ell k}{n-\ell k}\Bigr\rceil\ge \dfrac{\ell r+c}{c}= \dfrac{\ell r}{c}+1$. 
Since $\Bigl\lceil\dfrac{\ell k}{n-\ell k}\Bigr\rceil$ is an integer, we can conclude that $\dfrac{\ell k}{n-\ell k}> \dfrac{\ell r}{c}$, which gives $ck>r(n-\ell k)$, contradicting~\eqref{eq:compare_osz_2}.\linebreak\mbox{}\hfill\qedsymbol

\medskip
On the other hand, there are values of $n,k,r$ for which Theorem~\ref{thm-main} gives a better bound than Theorem~\ref{osz-thm}.
For instance if $n=137$, $k=56$ and $r=31$, then Theorem~\ref{thm-main} gives $\chi_{56q-31}(K(137,56))\ge137q-75$, whereas in Theorem~\ref{osz-thm} we need to take $\ell=2$ and $c>15.25$, which gives at best the bound $\chi_{56q-31}(K(137,56))\ge137q-77$.
And for $n=145$, $k=30$ and $r=17$, Theorem~\ref{thm-main} gives $\chi_{30q-17}(K(145,30))\ge145q-82$, whereas in Theorem~\ref{osz-thm} we need to take $\ell=3$ and $c>50$, or $\ell=4$ and $c>16.75$, which gives at best the bound $\chi_{30q-17}(K(145,30))\ge145q-140$.

\medskip
In the remainder of this subsection, we prove and discuss Theorem~\ref{thm-split}.

\begin{repthm}[\ref{thm-split}]\mbox{}\\*
    Let $n,k,r,q$ be integers, $n\ge2k\ge2$, $q\ge1$ and $0\le r\le k-1$.
    Choose $n_1,\ldots,n_t$ such that $n=\sum\limits_{i=1}^tn_i$ and $2k\le n_i<4k$ for all~$i$.
    Then we have $\chi_{qk-r}(K(n,k))\ge qn-\sum\limits_{i=1}^t\Bigl\lfloor\dfrac{n_ir}{k}\Bigr\rfloor$.
\end{repthm}

In fact, this theorem follows directly from the following more technical lemma.

\begin{mylemma}\label{lem-split}\mbox{}\\*
    Let $k,r,q$ be integers, $n\ge2k\ge2$, $q\ge1$ and $0\le r\le k-1$.
    Let $n_1,\ldots,n_t$ be positive integers such that $n=\sum\limits_{i=1}^tn_i$, and let $I_1,I_2,I_3\subseteq[t]$ be the sets of indices such that for $i\in I_1$ we have $n_i<k$, for $i\in I_2$ we have $k\le n_i<2k$, and for $i\in I_3$ we have $n_i\ge2k$.
    Then we have $\chi_{qk-r}(K(n,k))\ge |I_2|\cdot(qk-r)+\sum\limits_{i\in I_3}\Bigl(qn_i-\Bigl\lfloor\dfrac{n_ir}{k}\Bigr\rfloor\Bigr)$.
\end{mylemma}

The proof is based on the idea of partitioning $K(n,k)$ into suitable subgraphs by splitting the ground set $[n]$.
Note that for all~$m$ with $k\le m\le n-k$, the subgraph of $K(n,k)$ induced by $\Bigl\{F\in\dbinom{[n]}{k}\Bigm\vert F\subseteq[m]\Bigr\}$ is isomorphic to $K(m,k)$, and the subgraph of $K(n,k)$ induced by $\Bigl\{F\in\dbinom{[n]}{k}\Bigm\vert F\subseteq[m+1,n]\Bigr\}$ is isomorphic to $K(n-m,k)$.
Moreover, since the vertices of these two subgraphs are disjoint subsets of $[n]$, there is a complete bipartite join between them in $K(n,k)$.

Setting $k'=qk-r$, these observations immediately give for $k\le m\le n-k$:
\[\chi_{k'}(K(n,k))\ge\chi_{k'}(K(m,k))+\chi_{k'}(K(n-m,k)).\]
Moreover, for $m<k$ or $m>n-k$ we have $\chi_{k'}(K(m,k))=0$, which means that the inequality above in fact holds for $0\le m\le n$.

If $0\le n_i<k$, then $K(n_i,k)$ is empty, so $\chi_{qk-r}(K(n_i,k))=0$; while if $k\le n_i<2k$, then $K(n_i,k)$ has no edges, so $\chi_{qk-r}(K(n_i,k))= qk-r$.
Finally, for $n_i\ge2k$ Theorem~\ref{thm-main} gives $\chi_{qk-r}(K(n_i,k))\ge qn_i-\Bigl\lfloor\dfrac{n_ir}{k}\Bigr\rfloor$.
Combining it all we obtain
\[\chi_{qk-r}(K(n,k))\ge \sum_{i=1}^t \chi_{qk-r}(K(n_i,k))\ge |I_1|\cdot0+|I_2|\cdot(qk-r)+\sum\limits_{i\in I_3}\Bigl(qn_i-\Bigl\lfloor\dfrac{n_ir}{k}\Bigr\rfloor\Bigr),\]
completing the proof of the lemma.\hfill\qedsymbol

\medskip
Note that Lemma~\ref{lem-split} allows for any value of the $n_i$'s, whereas Theorem~\ref{thm-split} considers $2k\le n_i<4k$ only.
To justify this restriction, we show that the best bound in the lemma with $n=n_1+\cdots+n_t$ can always be obtained by taking $2k\le n_i<4k$.
First, if we have some $n_i$ with $n_i<k$, then since we know $\chi_{qk-r}(K(n_i,k))=0$, adding this $n_i$ to another $n_j$ never will give a worse bound.

Next assume we have some $n_i$ with $k\le n_i<2k$, but none with $n_i<k$.
If there are two such $n_i$, say $n_i$ and $n_j$, then it is easy to check that replacing them by $n_i+n_j$ (where $2k\le n_i+n_j<4k$) gives at least as good a bound.
If there is only one $n_i$ with $k\le n_i<2k$, then, since $n\ge2k$, there must be an $n_j$ with $n_j\ge2k$.
These two parts give $qk-r+qn_j-\Bigl\lfloor\dfrac{n_jr}{k}\Bigr\rfloor$ to the bound.
Replacing them by $n_i+n_j$, where we have $2k\le n_i+n_j$, replaces this contribution by $q(n_i+n_j)-\Bigl\lfloor\dfrac{(n_i+n_j)r}{k}\Bigr\rfloor$.
We can estimate
\begin{align*}
    q(n_i+n_j)-\Bigl\lfloor\dfrac{(n_i+n_j)r}{k}\Bigr\rfloor&\ge q(n_i+n_j)-\dfrac{(n_i+n_j)r}{k}\\
    &> \dfrac{n_i}{k}(qk-r)+qn_j-\Bigl\lfloor\dfrac{n_jr}{k}\Bigr\rfloor-1\ge qk-r+qn_j-\Bigl\lfloor\dfrac{n_jr}{k}\Bigr\rfloor -1.
\end{align*}
And since $q, k, r, n_i, n_j$ are all integers, we can conclude that $q(n_i+n_j)-\Bigl\lfloor\dfrac{(n_i+n_j)r}{k}\Bigr\rfloor\ge qk-r+qn_j-\Bigl\lfloor\dfrac{n_jr}{k}\Bigr\rfloor$, 
justifying the replacement.

Finally, assume we have some $n_i\ge4k$.
Then we can take $n_i=n_{i'}+n_{i''}$ with $n_{i'},n_{i''}\ge2k$.
Since $\Bigl\lfloor\dfrac{n_{i'}r}k\Bigr\rfloor+\Bigl\lfloor\dfrac{n_{i''}r}k\Bigr\rfloor\le\Bigl\lfloor\dfrac{(n_{i'}+n_{i''})r}k\Bigr\rfloor$, we find that splitting $n_i\ge4k$ always gives at least as good a bound.

\medskip
For many values of $n$ and $k$ there will be multiple ways to write $n=n_1+\cdots+n_t$ for $2k\le n_i<4i$.
Although intuitive one would expect that smaller values of~$n_i$ give better bounds because of the rounding in the term $\Bigl\lfloor\dfrac{n_ir}{k}\Bigr\rfloor$, this does not in general give the best bound.
For instance, if $n=76$, $k=7$ and $r=4$ (and $q$ can be anything), the choice $n_1=n_2=n_3=n_4=15$ and $n_5=16$, gives the bound $\chi_{7q-4}(K(76,7))\ge76q-41$.
But if we take $n_1=n_2=n_3=n_4=19$, we get $\chi_{7q-4}(K(76,7))\ge76q-40$.

Another intuitive idea is that the best bounds are obtained if the $n_i$ are similar in size (so they give roughly the same term $\Bigl\lfloor\dfrac{n_ir}{k}\Bigr\rfloor$.
Also that idea appeared to be wrong.
For instance, if $n=79$, $k=10$ and $r=6$ (again $q$ can be anything), the best bound is found by taking $n_1=23$ and $n_2=n_3=28$, or by taking $n_1=n_2=23$ and $n_3=33$.

We did quite extensive computations of bounds that can be obtained using Theorem~\ref{thm-split}, but weren't able to discover a pattern for what would be the optimal choice of the $n_i$'s for different values of $n,k,r$.

\subsection{Proof and Discussion of Theorem \ref{thm:main2}}

Theorem \ref{thm:main2} is a corollary of the following two results. 

\begin{mylemma}\label{thm:improve_stahl_large_q}\mbox{}\\* 
    For all $k\ge2$ and $n\ge2k+1$ there exist $q_0(n,k)$, such that if $q\ge q_0(n,k)+1$, then we have $\chi_{qk-r}(K(n,k))\ge n+\chi_{(q-1)k-r}(K(n,k))$ for all $0\le r\le k-1$. 
\end{mylemma}

\begin{mylemma}\label{thm:improve_CGJ_large_n}\mbox{}\\* 
    For all $k\ge2$ there exists $n_0(k)\le k^3-k^2+2k-2$, such that if $n\ge n_0(k)+1$, then we have $\chi_{qk-r}(K(n,k))\ge q+\chi_{qk-r}(K(n-1,k))$ for all $q\ge1$ and $0\le r\le k-1$.
\end{mylemma}

We first show how these lemmas provide Theorem~\ref{thm:main2}.

\begin{repthm}[\ref{thm:main2}]\mbox{}\\*
    Let $k\ge2$ be fixed.
    Then there exist $n_0(k)$ and $q_0(n,k)$ such that the following holds.
    If for all $2k\le n\le n_0(k)$ we know that $\chi_{qk-(k-1)}(K(n,k))=qn-2(k-1)$ for at least one $q\ge q_0(n,k)$, then we have $\chi_{qk-r}(K(n,k))=qn-2r$ for all $n\ge2k$, $q\ge1$ and $0\le r\le k-1$.
\end{repthm}

\begin{proof}
Fix $k\ge2$.
Since Stahl's conjecture holds for $n=2k$, we can assume $n\ge2k+1$.
Let $n_0(k)$ be the integer as in Lemma~\ref{thm:improve_CGJ_large_n}. 
So if $n\ge n_0(k)+1$, then for all $0\le r\le k-1$ we have 
\begin{equation}\label{eq:proof_main2_0}
    \chi_{qk-r}(K(n,k))\ge q+\chi_{qk-r}(K(n-1,k)).
\end{equation}
For each $n\ge2k+1$, let $q_0(n,k)$ be the integer as in Lemma~\ref{thm:improve_stahl_large_q}.
So if $q\geq q_0(n,k)+1$, then for all $0\le r\le k-1$ we have 
\begin{equation}\label{eq:proof_main2_1}
    \chi_{qk-r}(K(n,k))\ge n+\chi_{(q-1)k-r}(K(n,k)).
\end{equation}
Next note that for all $n\ge2k$ and $q'\ge 2$ we have 
\begin{equation}\label{eq:proof_main2_2}
    \chi_{q'k-r}(K(n,k))\le n+\chi_{(q'-1)k-r}(K(n,k)),
\end{equation}
since combining a $\bigl(\chi_{(q'-1)k-r}(K(n,k)),(q'-1)k-r\bigr)$-colouring and a $(n,k)$-colouring of $K(n,k)$ produces a $(q'k-r)$-multi-colouring.

Assume the conditions in Theorem~\ref{thm:main2} hold.
I.e.\ for each $2k\le n\le n_0(k)$ there is a $q_n\ge q_0(n,k)$ such that $\chi_{q_nk-(k-1)}(K(n,k))= q_nn-2(k-1)$.
Then we immediately have $\chi_{q_nk-r}(K(n,k))= q_nn-2r$ for all $0\le r\le k-1$, since for any non-empty graph $G$ we have $\chi_{k'+1}(G)\ge \chi_{k'}(G)+2$ (by Lemma~\ref{sta:prop}), and $\chi_{q_nk}(K(n,k))= q_nn$. 
Combining this with \eqref{eq:up_bd}, \eqref{eq:proof_main2_1}, and \eqref{eq:proof_main2_2}, gives $\chi_{qk-r}(K(n,k))= qn-2r$ for all $2k\le n\le n_0(k)$, $q\ge1$ and $0\le r\le k-1$.

In particular we have that $\chi_{qk-r}(K(n,k))= qn-2r$ for $n=n_0(k)$ and all $q\ge1$ and $0\le r\le k-1$.
Combining this with \eqref{eq:up_bd} and \eqref{eq:proof_main2_0} gives $\chi_{qk-r}(K(n,k))= qn-2r$ for all $n\ge n_0(k)$, $q\ge1$ and $0\le r\le k-1$, completing the proof.
\end{proof}

Essential in the proofs of Lemmas~\ref{thm:improve_stahl_large_q} and~\ref{thm:improve_CGJ_large_n} is some more detailed information about large independent sets in Kneser graphs.
Recall that for all $n\ge2k$ the independence number of the Kneser graph $K(n,k)$ is $\alpha(K(n,k))=\dbinom{n-1}{k-1}$~\cite{erdos_intersection_1961}, and for $n\ge2k+1$ equality only occurs for \emph{trivial} independent sets: vertex sets whose vertices are $k$-sets in $[n]$ that contain some fixed common element $i\in[n]$.
We say that such an independent set is \emph{centred at $i$}.

Hilton and Milner~\cite{hilton1967some} showed that if $n\ge 2k+1$ and an independent set in the Kneser graph $K(n,k)$ is not trivial, then it has order at most $\dbinom{n-1}{k-1}-\dbinom{n-k-1}{k-1}+1$.
This `second best' bound is significantly smaller than the Erd{\H{o}}s-Ko-Rado bound, which means that for large~$n$ and~$q$, many of the colours used in a `good' $(qk-r)$-multi-colouring of $K(n,k)$ must induce trivial independent sets.
In the remainder, we use $\alpha^*(K(n,k))=\dbinom{n-1}{k-1}-\dbinom{n-k-1}{k-1}+1$ for the Hilton-Milner bound.

\begin{proof}[Proof of Lemma~\ref{thm:improve_stahl_large_q}]
Fix some $k\ge 2$, $n\ge 2k+1$ and $0\le r\le k-1$.
Take $q\ge q_0(n,k)+1$, where $q_0(n,k)$ is a function that we will specify later.
Fix an $(x,qk-r)$-colouring $C$ of $K(n,k)$ for some $x\le qn-2r$ (which we know exists by~\eqref{eq:up_bd}).
We will show this means there also exists an $(x-n, (q-1)k-r)$-colouring of $K(n,k)$. 

Let $y$ be the number of \emph{non-trivial} colour classes in $C$ (the colour classes that are not a subset of $\Bigl\{F\in\dbinom{[n]}{k}\Bigm\vert i\in F\Bigr\}$ for some $i\in[n]$).
Hence there are $x-y$ \emph{trivial} colour classes.
By counting the appearance of each vertex in all colour classes, we have
\allowdisplaybreaks
\begin{align*}
    (qk-r)\binom{n}{k}&\le (x-y)\alpha(K(n,k))+y\alpha^*(K(n,k))\\
    &\le (qn-2r-y)\binom{n-1}{k-1}+y\Bigl(\binom{n-1}{k-1}-\binom{n-k-1}{k-1}+1\Bigr).
\end{align*} 
Since $(qk-r)\dbinom{n}{k}=\Bigl(qn-\dfrac{rn}{k}\Bigr)\dbinom{n-1}{k-1}$, this gives
\begin{align}\label{eq:improve_stahl_1}
    y\le \frac{(\frac{rn}{k}-2r)\binom{n-1}{k-1}}{\binom{n-k-1}{k-1}-1}= \frac{r(n-2k)\binom{n-1}{k-1}}{k\Bigl(\binom{n-k-1}{k-1}-1\Bigr)}.
\end{align}

We claim that if $q\ge q_0(n,k)+1$, where 
\begin{equation}\label{def-q0}
    q_0(n,k):= \biggl\lfloor\frac{(k-1)(n-2k+1)}{n-k}+ \frac{(k-1)(n-2k)(n-1)\Bigl(\binom{n-2}{k-1}-\binom{n-k-1}{k-1}\Bigr)}{k(n-k)\Bigl(\binom{n-k-1}{k-1}-1\Bigr)}\biggr\rfloor,
\end{equation}
then for all $i\in [n]$ there is a \emph{trivial} colour class in~$C$ that is \emph{centred} at $i$.
For a contradiction, assume that there is no colour class centred at $i^*$ for some $i^*\in [n]$, i.e.\ none of the colour classes in $C$ is a subset of $\mathcal{F}= \Bigl\{F\in \dbinom{[n]}{k}\Bigm\vert i^*\in F\Bigr\}$.
Then each trivial colour class contains at most $\dbinom{n-2}{k-2}$ vertices in~$\mathcal{F}$, and each non-trivial colour class contains at most $\dbinom{n-1}{k-1} - \dbinom{n-k-1}{k-1}$ vertices in~$\mathcal{F}$. 
Counting the appearance of each vertex in $\mathcal{F}$ in all colour classes, we find
\begin{align*}
    (qk-r)|\mathcal{F}|&\le (x-y)\binom{n-2}{k-2}+y\Bigl(\binom{n-1}{k-1}-\binom{n-k-1}{k-1}\Bigr)\\
    &\le (qn-2r)\binom{n-2}{k-2}+y\Bigl(\binom{n-1}{k-1}-\binom{n-2}{k-2}-\binom{n-k-1}{k-1}\Bigr).
\end{align*}
Since $(qn-r)|\mathcal{F}|=(qn-r)\dbinom{n-1}{k-1}=\dfrac{(qk-r)(n-1)}{k-1}\dbinom{n-2}{k-2}$ and $\dbinom{n-1}{k-1}-\dbinom{n-2}{k-2}=\dbinom{n-2}{k-1}$, this gives
\begin{equation}\label{eq:improve_stahl_2}
    y\Bigl(\binom{n-2}{k-1}-\binom{n-k-1}{k-1}\Bigr)\ge \frac{q(n-k)-r(n-2k+1)}{k-1}\binom{n-2}{k-2}.
\end{equation}
Combining \eqref{eq:improve_stahl_1} and \eqref{eq:improve_stahl_2}, we obtain
\[\frac{r(n-2k)\binom{n-1}{k-1}\Bigl(\binom{n-2}{k-1}-\binom{n-k-1}{k-1}\Bigr)}{k\Bigl(\binom{n-k-1}{k-1}-1\Bigr)}\ge \frac{q(n-k)-r(n-2k+1)}{k-1}\binom{n-2}{k-2} .\]
We can rearrange this to
\[q\le \frac{r(n-2k+1)}{n-k}+ \frac{r(n-2k)(n-1)\Bigl(\binom{n-2}{k-1}-\binom{n-k-1}{k-1}\Bigr)}{k(n-k)\Bigl(\binom{n-k-1}{k-1}-1\Bigr)}. \] 
Since $r\le k-1$, the right-hand side of this inequality is smaller than $q_0(n,k)+1$ as defined in~\eqref{def-q0}, a contradiction.
 
So we can assume that for all $i\in [n]$ there is a trivial colour class in $C$ that is centred at $i$.
By removing one such trivial colour class centred at $i$ for each $i\in [n]$, we remove $n$ colour classes in total and at most $k$ colours for each vertex.
This gives an $(x-n,(q-1)k-r)$-colouring of $K(n,k)$, thus proving $\chi_{(q-1)k-r}(K(n,k))\le \chi_{qk-r}(K(n,k))-n$.
\end{proof} 

We need the following technical result for the proof of Lemma~\ref{thm:improve_CGJ_large_n}.

\begin{mylemma}\label{lm:compare_indep}\mbox{}\\* 
For all $k\ge2$, if $n_0(k):=k^3-k^2+2k-2$ then for $n\ge n_0(k)+1$ we have $\dfrac{\alpha^*(K(n,k))}{\alpha(K(n,k))}< \dfrac{n}{k(n-2k+2)}$.
\end{mylemma}

\begin{proof}
We can write
\[\frac{\alpha^*(K(n,k))}{\alpha(K(n,k))}= \frac{\binom{n-1}{k-1}-\binom{n-k-1}{k-1}+1}{\binom{n-1}{k-1}}= 1-\prod_{i=1}^{k-1}\frac{n-k-i}{n-i} +\frac{1}{\binom{n-1}{k-1}}.\]
We first estimate $-\prod\limits_{i=1}^{k-1}\dfrac{n-k-i}{n-i}= -\prod\limits_{i=1}^{k-1}\Bigl(1-\dfrac{k}{n-i}\Bigr)\le -\Bigl(1-\dfrac{k}{n-k+1}\Bigr)^{k-1}$. 

Next, it is straightforward to check that if $k=2$, then we have $\dfrac{1}{\binom{n-1}{k-1}}=\dfrac{1}{n-1}<\dfrac{1}{n-2}=\dfrac{2k-2}{k(n-2k+2)}$ for all $n\ge3$.
For $k\ge2$ and $n\ge2k+1$, we can estimate by induction on~$k$:
\begin{align*}
    \frac{1}{\binom{n-1}{(k+1)-1}}= \frac{1}{\frac{n-k}{k}\binom{n-1}{k-1}}&<\frac{k}{n-k}\cdot\frac{2k-2}{k(n-2k+2)}\\
    &= \frac{(2k-2)(k+1)(n-2k)}{2k(n-k)(n-2k+2)}\cdot\frac{2(k+1)-2}{(k+1)(n-2(k+1)+2)}\\
    &< \frac{2(k+1)-2}{(k+1)(n-2(k+1)+2)}.
\end{align*}
This shows that $\dfrac{1}{\binom{n-1}{k-1}}<\dfrac{2k-2}{k(n-2k+2)}$ for all $n\ge2k+1\ge5$.

All in all, this means that it suffices to find $n_0(k)\ge2k+1$ such that for all $n\ge n_0(k)+1$ we have
\[1-\Bigl(1-\dfrac{k}{n-k+1}\Bigr)^{k-1}+\dfrac{2k-2}{k(n-2k+2)}\le \dfrac{n}{k(n-2k+2)}.\]
This inequality is equivalent to $\Bigl(1-\dfrac{k}{n-k+1}\Bigr)^{k-1}\ge 1-\dfrac{1}{k}$, hence to
\[(k-1)\ln\Bigl(1-\dfrac{k}{n-k+1}\Bigr)\ge \ln\Bigl(1-\dfrac{1}{k}\Bigr).\]
Now we use the standard inequalities $1-\dfrac{1}{x}\le \ln(x)\le x-1$ to obtain that it is enough to guarantee
$(k-1)\cdot\dfrac{-k}{n-2k+1}\ge \dfrac{-1}{k}$.
This holds for $n\ge k^3-k^2+2k-1$, completing the proof.
\end{proof}

Note that for specific values of~$k$ we can get better bounds on $n_0(k)$.
For instance, computations show that for $k=4$ the conclusion of the lemma already holds for $n\ge39$ (whereas $n_0(4)+1=55$).

Now we are ready to prove Lemma~\ref{thm:improve_CGJ_large_n}.

\begin{proof}[Proof of Lemma~\ref{thm:improve_CGJ_large_n}]
Take  $k\ge2$ and $n\ge n_0(k)+1$, with $n_0(k)$ as in Lemma~\ref{lm:compare_indep} (hence definitely $n\ge2k+1$), and assume there exists an $(x, qk-r)$-colouring of $K(n,k)$ for some $x\le qn-2r$. We will prove there is an $(x-q,qk-r)$-colouring of $K(n-1,k)$, which shows $\chi_{qk-r}(K(n-1,k))\le \chi_{qk-r}(K(n,k))-q$.

We first claim there are at least $(q-1)n+1$ trivial colour classes in the $(x,qk-r)$-colouring of $K(n,k)$. If this is not the case, then there are at most $(q-1)n$ colour classes that appear on more than $\alpha^*(K(n,k))$ vertices. Hence counting the total number of appearance of each vertex in all colour classes, we have
\begin{align*}
    (qk-r)\binom{n}{k}& \le (q-1)n\alpha(K(n,k)) + (x-(q-1)n)\alpha^*(K(n,k))\nonumber\\
    &\le (q-1)n\alpha(K(n,k)) + (qn-2r-(q-1)n)\alpha^*(K(n,k)).
\end{align*}
Since $\dbinom{n}{k}= \dfrac{n}{k}\dbinom{n-1}{k-1}= \dfrac{n}{k}\alpha(K(n,k))$, we can rearrange this to
\[\dfrac{\alpha^*(K(n,k))}{\alpha(K(n,k))}\ge \dfrac{n(k-r)}{k(n-2r)}= \dfrac{n}{2k}\Bigl(1-\dfrac{n-2k}{n-2r}\Bigr)\ge \dfrac{n}{k(n-2k+2)},\]
where we use that $r\le k-1$.

This contradicts Lemma~\ref{lm:compare_indep}.
Hence there are at least $(q-1)n+1$ trivial colour classes in the $(x, qk-r)$-colouring, where each trivial colour class is a subset of $\Bigl\{F\in \dbinom{[n]}{k}\Bigm\vert i\in F\Bigr\}$ for some $i\in[n]$.
Therefore there is some $i^*\in [n]$ such that at least $q$ trivial colour classes are subsets of $\Bigl\{F\in \dbinom{[n]}{k}\Bigm\vert i^*\in F\Bigr\}$.
Removing those $q$ trivial colour classes, we obtain an $(x-q,qk-r)$-colouring of $K(n-1,k)$, as required.
\end{proof}

As we already know that Stahl's Conjecture is true for $n=2k$ and $n=2k+1$, Theorem~\ref{thm:main2} shows that for every $k$, at most $k^3-k^2-3$ values of $\chi_{k'}(n,k)$ need to be determined (one for each $n$, $2k+2\le n\le n_0(k)$) to prove the conjecture for that value of $k$ and all $n,q$ (or find a counterexample).

We next prove the upper bound of $q_0(n,k)$ mentioned in Section~\ref{multi_sec:intro}.

\begin{myprop}\label{q_0-bound}\mbox{}\\*
For all $n\ge 2k+1$ we have $q_0(n,k)< 4^k(n-2k)$.
\end{myprop}

\begin{proof}
We first estimate, using the definition of $q_0(n,k)$ in~\eqref{def-q0},
\begin{align*}
    q_0(n,k)&\le \frac{(k-1)(n-2k+1)}{n-k}+ \frac{(k-1)(n-2k)(n-1)\Bigl(\binom{n-2}{k-1}-\binom{n-k-1}{k-1}\Bigr)}{k(n-k)\Bigl(\binom{n-k-1}{k-1}-1\Bigr)}\\
    &= \frac{(k-1)(n-2k)(n-1)}{k(n-k)}\biggr(\frac{\binom{n-2}{k-1}-1}{\binom{n-k-1}{k-1}-1}-1\biggr) + \frac{(k-1)(n-2k+1)}{n-k}\\
    &= \frac{(k-1)(n-2k)(n-1)}{k(n-k)}\cdot\frac{\binom{n-2}{k-1}-1}{\binom{n-k-1}{k-1}-1} - \frac{(k-1)(n-2k-1)}{k}\\
    &= \frac{(k-1)(n-2k)(n-1)}{k(n-k)}\biggl(\frac{\binom{n-2}{k-1}}{\binom{n-k-1}{k-1}} + \frac{1}{\binom{n-k-1}{k-1}-1}\Bigl(\frac{\binom{n-2}{k-1}}{\binom{n-k-1}{k-1}}-1\Bigr)\biggr) \\
    &\phantom{{}={}}- \frac{(k-1)(n-2k-1)}{k}\\
    &\le \frac{(k-1)(n-2k)(n-1)}{k(n-k)}\biggl(\frac{\binom{n-2}{k-1}}{\binom{n-k-1}{k-1}} + \frac{\binom{n-2}{k-1}}{(k-1)\binom{n-k-1}{k-1}} - \frac{1}{k-1}\biggr) \\
    &\phantom{{}={}}- \frac{(k-1)(n-2k-1)}{k}\\
    &= \frac{(n-2k)(n-1)}{(n-k)}\cdot \frac{\binom{n-2}{k-1}}{\binom{n-k-1}{k-1}} - (n-2k) + \frac{k-1}{n-k}\\
    &= (n-2k)\prod_{i=1}^{k-1} \frac{n-i}{n-k-i}  - (n-2k) + \frac{k-1}{n-k}\\
    &< (n-2k)\prod_{i=1}^{k-1} \frac{n-i}{n-k-i}.
\end{align*}
In the second inequality we used that $\dbinom{n-2}{k-1}\ge \dbinom{n-k-1}{k-1}\ge \dbinom{2k+1-k-1}{k-1}= k$.

We bound $\prod\limits_{i=1}^{k-1}\dfrac{n-i}{n-k-i}$ by estimating its logarithm for $n\ge 2k+1$:
\begin{align*}
    \ln{\Bigl(\prod_{i=1}^{k-1} \frac{n-i}{n-k-i}\Bigr)} 
    = \sum_{i=n-2k+1}^{n-k-1}\!\! \ln{\Bigl(1+\frac{k}{i}\Bigr)}
    &< \int_{n-2k}^{n-k-1} \ln{\Bigl(1+\frac{k}{x}\Bigr)}\,\mathrm{d}x \\
    &= (n-1)\ln{(n-1)} + (n-2k)\ln{(n-2k)} \\
    &\phantom{{}={}} - (n-k-1)\ln{(n-k-1)} - (n-k)\ln{(n-k)}.
\end{align*}
We will estimate this by setting, for $x\ge2k+1$:
\[f(x):= (x-1)\ln(x-1)+(x-2k)\ln(x-2k)-(x-k-1)\ln(x-k-1)-(x-k)\ln(x-k).\]
Then we have $f'(x)= \bigl(\ln(x-1)-\ln(x-k-1)\bigr)-\bigl(\ln(x-k)-\ln(x-2k)\bigr)$.
Now set $g(y):=\ln y-\ln(y-k)$, for $y\ge k+1$.
It's straightforward to check that $g'(y)<0$ for $y\ge k+1$. Hence $g(y)$ is strictly decreasing and so for $x\ge2k+1$ we have $f'(x)= g(x-1)-g(x-k)<0$.
So also $f(x)$ is strictly decreasing; which means that $f(x)\le f(2k+1)$ for all $x\ge2k+1$.
Using this we obtain
\begin{align*}
    \ln\Bigl(\prod_{i=1}^{k-1} \frac{n-i}{n-k-i}\Bigr)
    &< 2k\ln(2k) + \ln1 - k\ln k - (k+1)\ln(k+1)\\[-1mm]
    &= 2k\ln2- \bigl((k+1)\ln(k+1)- k\ln k\bigr)< 2k\ln2.
\end{align*}
We can conclude that $\prod\limits_{i=1}^{k-1}\dfrac{n-i}{n-k-i}< \mathrm{e}^{2k\ln{2}}= 4^k$, which gives $q_0(n,k)< 4^k(n-2k)$.
\end{proof}

Note that for specific values of $n,k$ we can get better bounds on $q_0(n,k)$.
For instance, if $k=4$ and $n\ge11$, direct computation shows $q_0(n,4) = 12$.
Hence for each $n\ge11$ it suffices to find $\chi_{12\cdot4-(4-1)}(K(n,4))=\chi_{45}(K(n,4))$ to decide Stahl's Conjecture for $k=4$ and those~$n$.

For fixed $k$, larger $n$ can give better bounds for the expression $\prod\limits_{i=1}^{k-1}\dfrac{n-i}{n-k-i}=\!\prod\limits_{i=n-2k+1}^{n-k-1}\!\dfrac{k+i}{i}$ that appears in the proof above, which in its turn gives smaller $q_0(n,k)$.
For instance,

(a)\mbox{ \ }\textit{if $n\ge ck-1$ for some constant $c>2$, then $\prod\limits_{i=n-2k+1}^{n-k-1}\!\dfrac{k+i}{i}\le \Bigl(\dfrac{c-1}{c-2}\Bigr)^{k-1}\!$, and hence\linebreak
\phantom{(a)\mbox{ \ }}$q_0(n,k)< (n-2k)\Bigl(\dfrac{c-1}{c-2}\Bigr)^{k-1}$;}

(b)\mbox{ \ }\textit{if $n\ge k^2+k-1$, then $\prod\limits_{i=n-2k+1}^{n-k-1}\!\dfrac{k+i}{i}\le \Bigl(\dfrac{k}{k-1}\Bigr)^{k-1}\!< \mathrm{e}$, and hence $q_0(n,k)<\mathrm{e}(n-2k)$};

(c)\mbox{ \ }\textit{if $n\ge k^3+k-1$, then $\prod\limits_{i=1}^{k-1}\dfrac{n-i}{n-k-i}\le \Bigl(\dfrac{k^2}{k^2-1}\Bigr)^{k-1}\!< 1+\dfrac{1}{k}$, and hence $q_0(n,k)<{}$\linebreak
\phantom{(c)\mbox{ \ }}$\dfrac{k+1}{k}(n-2k) - (n-2k)+\dfrac{k-1}{n-k} < \dfrac{n}{k}$.}

\medskip
Note that if Stahl's Conjecture is false for some $n,q,k,r$, then Lemmas~\ref{thm:improve_stahl_large_q} and~\ref{thm:improve_CGJ_large_n} immediately give an infinite number of counterexamples.

\begin{mycor}\label{cor:comb_improve_large_n_q}\mbox{}\\*
    If for some $n,q,k,r$, $n\ge2k+1\ge5$, $q\ge1$, $0\le r\le k-1$, we have $\chi_{qk-r}(K(n,k))\le qn-2r-1$, then for all $r',q'$ with $r\le r'\le k-1$, $q'\ge q+r'-r$ we have $\chi_{q'k-r'}(K(n,k))\le q'n-2r'-1$.
\end{mycor}

\section{Concluding Remarks}\label{multi_sec:conclude}

Our starting point for the research describe in this note was Question~\ref{main_ques}, which we felt was a quite natural question, and then discovered was equivalent to a conjecture made by Stahl in 1976.
We consider as our main contribution the simple observation that allowed us to give short proofs for many of the known cases of Stahl's Conjecture; and in fact sometimes extend those cases.

On the other hand, we have no illusion that similar observations will be enough to prove the conjecture.
For the most part it remains stubbornly open.

\medskip
Stahl's Conjecture is a wide-ranging generalisation of the original Kneser Conjecture.
Csorba and Oszt\'enyi~\cite{csorba2010topological} showed that the topological lower bounds used by Lov\'asz to prove the Kneser Conjecture cannot be used on their own to prove Stahl's Conjecture.
They proved that for $k'\ge\dbinom{n}{k}$, that topological lower bound implies $\chi_{k'}(K(n,k))\ge k'\Bigl\lfloor \dfrac{n}{k}\Bigr\rfloor$ only.

On the other hand, our bounds in (the proofs of) Lemmas~\ref{thm:improve_stahl_large_q} and~\ref{thm:improve_CGJ_large_n} show that for each $n\ge 2k+1$, we only need to determine one $\chi_{k'}(K(n,k))$ with $k'=qk-(k-1)$, where $q\le q_0(n,k)$ can be quite small for large~$n$ compared to~$k$.
Hence it might still be possible that topological methods similar to those Lov\'asz used can prove Stahl's Conjecture for large enough $n$ (depending on~$k$).

\medskip
We like to end with a question about multi-colouring that can be seen as a ``symmetric'' version of Stahl's Conjecture, and which we felt was (almost) equally natural as Question~\ref{main_ques}: ``\emph{For what pairs $(n,k)$ and $(n',k')$ is it the case that every $(n,k)$-colourable graph is also $(n',k')$-colourable and every $(n',k')$-colourable graph is also $(n,k)$-colourable?}''.
The similar question for normal colouring is again trivial to answer: ``\emph{For all $n,n'$ we have that every $n$-colourable graph is $n'$-colourable and every $n'$-colourable graph is $n$-colourable, if and only if $n=n'$}''.

In this case we can give the full answer for multi-colouring as well, although to answer it we needed some knowledge about special cases of Stahl's Conjecture.

\begin{mythm}\label{thm:homo_both_way}\mbox{}\\*
For all integers $n,k,n',k'$, $n\ge2k\ge2$, $n'\ge2k'\ge2$, every $(n,k)$-colourable graph is also $(n',k')$-colourable and every $(n',k')$-colourable graph is also $(n,k)$-colourable, if and only if\\
{\rm (1)} $n=n'$ and $k=k'$, or\\
{\rm (2)} $n=2k$ and $n'=2k'$.
\end{mythm}

\begin{proof} 
Translated into homomorphisms, we are asking for pairs $(n,k)$ and $(n',k')$ such that there is both a homomorphism $K(n,k)\rightarrow K(n',k')$ and a homomorphism $K(n',k')\rightarrow K(n,k)$

If $n=n'$ and $k=k'$, then clearly there are homomorphisms $K(n,k)\rightarrow K(n',k')$ and $K(n',k')\rightarrow K(n,k)$.\\
If $n=2k$ and $n'=2k'$, then both $K(n,k)$ and $K(n',k')$ are bipartite.
We can easily map any bipartite graph to any graph with at least one edge, by mapping all vertices in one part of the bipartition to one endvertex of an edge, and all vertices in the other part to the other endvertex of that same edge.

For the reverse implication, assume that we have homomorphisms $K(n,k)\rightarrow K(n',k')$ and $K(n',k')\rightarrow K(n,k)$.
This means that $\chi_{k'}(K(n,k))\le n'$ and $\chi_k(K(n',k'))\le n$.
On the other hand, Observation~\ref{obs1} gives $\chi_{k'}(K(n,k))\ge\dfrac{k'n}{k}$ and $\chi_k(K(n',k'))\ge\dfrac{kn'}{k'}$.
So we have both $\dfrac{k'n}{k}\le n'$ and $\dfrac{kn'}{k'}\le n$.
This leads to the equality $\dfrac{n'}{k'}=\dfrac{n}{k}$ and in particular gives that both $\chi_{k'}(K(n,k))=\dfrac{k'n}{k}$ and $\chi_k(K(n',k'))=\dfrac{kn'}{k'}$.
If conclusion~(2) does not hold, we must have $\dfrac{n'}{k'}=\dfrac{n}{k}>2$.
We obtain $n'= \dfrac{k'n}{k}>2k'$ and $n=\dfrac{kn'}{k'}>2k$. 
Then Theorem~\ref{thm:eq_indep_bound} gives that~$k'$ is a multiple of $k$ and $k$ is a multiple of $k'$.
That means $k=k'$, hence also $n=n'$, and so conclusion~(1) holds.
\end{proof}

\small
\bibliographystyle{abbrv}
\bibliography{references.bib}
\end{document}